\documentclass[onecolumn, draftcls]{IEEEtran}

\usepackage{amsbsy}
\usepackage{amsmath}
\usepackage{nomencl}
\usepackage{mathtools}
\usepackage{amsfonts} 
\usepackage{amssymb}  
\usepackage{amsthm}   
\usepackage{amsxtra}
\usepackage{bm}
\usepackage{multirow}
\usepackage{color,colortbl}
\usepackage{dsfont}
\usepackage{enumerate}
\usepackage{fancybox}
\usepackage{fancyhdr}
\usepackage{graphics}
\usepackage{graphicx}
\usepackage{latexsym}
\usepackage{subfig}
\usepackage{url}
\graphicspath{{./image_jv1/}}

\newtheorem{thm}{Theorem}

\definecolor{Dblue}{rgb}{0,0,1}
\definecolor{Dbrown}{rgb}{0.59,0.4,0}
\definecolor{Dred}{rgb}{0.64,0,0}
\definecolor{Dgreen}{rgb}{0,0.4,0}

\def \bs {\boldsymbol}
\def \mS {\mathcal{S}}
\def \mC {\mathcal{C}}
\def \mH {\mathcal{H}}
\def \rs {\mathrm{s}}
\def \rr {\mathrm{r}}
\def \rd {\mathrm{d}}
\def \rb {\mathrm{b}}
\def \rh {\mathrm{h}}
\def \rp {\mathrm{p}}
\def \rw {\mathrm{w}}
\def \rcap {\mathrm{cap}}
\def \r {\bs r}
\def \rinv {\mathrm{inv}}
\def \rom {\mathrm{o/m}}
\def \Ss {S^\rs}
\def \mR {\mathcal{R}}
\def \Pr {\mathrm{Pr}}
\def \mJ {\mathcal{J}}
\def \mT {\mathcal{T}}
\def \x  {\bs x}
\def \z  {\bs z}
\def \v  {\bs v}
\def \mC {\mathcal{C}}
\def \mB {\mathcal{B}}
\def \tbz {\tilde{\bs z}}

\IEEEoverridecommandlockouts

\makenomenclature

\begin{document}

\title{Joint Optimization of Hybrid Energy Storage and Generation Capacity with Renewable Energy}
\author{Peng~Yang,~\IEEEmembership{Student Member,~IEEE},
        and~Arye~Nehorai,~\IEEEmembership{Fellow,~IEEE}
\thanks{The authors are with the Preston M. Green Department
of Electrical and Systems Engineering, Washington University in St. Louis, St. Louis,
MO, 63130 USA. E-mail: yangp@ese.wustl.edu, nehorai@ese.wustl.edu.}
\thanks{This work was supported by the International Center for Advanced Renewable Energy and Sustainability (I-CARES) at Washington University in St. Louis.}
}

\maketitle

\begin{abstract}
In an isolated power grid or a micro-grid with a small carbon footprint, the penetration of renewable energy is usually high. In such power grids, energy storage is important to guarantee an uninterrupted and stable power supply for end users. Different types of energy storage have different characteristics, including their round-trip efficiency, power and energy rating, energy loss over time, and investment and maintenance costs. In addition, the load characteristics and availability of different types of renewable energy sources vary in different geographic regions and at different times of year. Therefore joint capacity optimization for multiple types of energy storage and generation is important when designing this type of power systems. In this paper, we formulate a cost minimization problem for storage and generation planning, considering both the initial investment cost and operational/maintenance cost, and propose a distributed optimization framework to overcome the difficulty brought about by the large size of the optimization problem. The results will help in making decisions on energy storage and generation capacity planning in future decentralized power grids with high renewable penetrations.
\end{abstract}


\printnomenclature[0.55in]

\section{Introduction}
Renewable energy sources \cite{re:renewable}, including solar and wind energy, provide only about 3\% of the electricity in the United States. However, high penetration of renewable energy is becoming the trend for various reasons. The projected future shortage in fossil fuels, environmental concerns, and advances in smart grid \cite{sg:collier10} technologies stimulate the increasing penetration of renewable energy. Researchers have shown that supplying all the energy needs of the United States from renewable energy is realizable in the future \cite{sg:Turner99}. According to the National Renewable Energy Laboratory (NREL), renewable energy potentially will support about 80\% of the total electricity consumption in the U.S. in 2050 \cite{re:nrel12}. The high penetration of renewable energy is especially common in (remote) isolated grids, or micro-grids with small carbon footprints \cite{sg:wiedmann08}, which are self-sustained most of the time.

Most renewable energy sources, including wind and solar, are highly intermittent. The availability of such energy sources varies significantly in different geographical locations. In the same location, the amount of generation also fluctuates depending on the time of day, season, and weather conditions. A grid with high renewable energy penetration needs to build sufficient energy storage to ensure an uninterrupted supply to end users \cite{sg:carrasco06},\cite{sg:Ibrahim08}.
There are different types of energy storages, including super-capacitors, flywheels, chemical batteries, water pumps, hydrogen, and compressed air \cite{sg:gonzalex08}-\nocite{sg:trong10}\nocite{sg:zhou11ces}\nocite{sg:Chalk06}\nocite{sg:barton04}\cite{sg:vosburgh78}. Different types of energy storage have different characteristics, e.g., round-trip energy efficiency, maximum capacity/power rating, energy loss over time, and investment/operational costs. For example, flywheel energy storage has high energy efficiency and charge/discharge rates, but the rate of energy loss over time is relatively high. Chemical batteries have relatively high energy efficiency and low energy loss over time, however their maintenance cost is high due to their low durability, which is quantified by cycling capacity\footnote{The maximum number of charging cycles (full charge and discharge).}. Water pumps and hydrogen energy storages have low energy efficiency, but their energy loss over time is small. Therefore they are often used for longer-term energy storage.

Although there has been research on planning and/or operating a specific type of energy storage system for isolated electricity grids \cite{sg:abbey07}-\nocite{sg:brown08}\nocite{sg:abbey09}\cite{sg:zhang13}, few works consider exploiting the different characteristics of multiple types of energy storage and the different availabilities of multiple types of renewable energy sources, forming a hybrid energy generation and storage system. Nevertheless, jointly planning for energy storage along with renewable generation capacity potentially results in a more economical and efficient energy system.

Since the future grid is becoming decentralized, we consider the scenario of an isolated grid, or a micro-grid with a small carbon footprint, whose energy is generated mainly from renewable energy sources. To make the scenario more practical, we assume the grid also has traditional diesel generators. The diesel generator on its own is insufficient to supply the demand of the grid, as its generation capacity is significantly less than the peak load. We formulate an optimization problem with the objective of minimizing the investment cost and operational/maintenance cost of energy storage and generators, by finding an optimal combination of different energy storages and generators (which we refer to as design parameters) and optimizing their operations.

The renewable generation and user demands change with time, and have different characteristics at different times of day and different days of the year. It is often difficult to obtain an accurate probability density function to reflect these complex characteristics. Therefore, several years of historical data may be needed to obtain better optimization results. As the size of historical database increases, the design horizon of the optimization problem increases, and the problem becomes increasingly difficult to solve. To resolve this problem, we reformulate the original problem as a consensus problem. The entire design horizon is divided into multiple shorter horizons, and thus the design parameters become the consensus parameters, which should be consistent across all sub-problems. This framework can also be extended to the case of solving chance-constrained optimization using scenario approximations, as we will elaborate later. We propose to solve the consensus problem in a parallel distributed manner based on the alternating direction method of multipliers (ADMM) \cite{opt:admm10}, which mitigates the curse of dimensionality due to increased number of scenarios.

The rest of this paper is organized as follows. In Sec.\ \ref{sec:related} we briefly review some relevant works. In Sec.\ \ref{sec:model} we describe the system model, including the energy storage and generators. In Sec.\ \ref{sec:opt} we formulate the optimization problem and solve it in a distributed manner. We provide numerical examples in Sec.\ \ref{sec:exp}, and conclude the paper in Sec.\ \ref{sec:conclusion}.

{\em Notations:} We use italic symbols to denote scalars, bold italic symbols to denote vectors, calligraphic symbols to denote sets, $\mathrm{card}(\cdot)$ to denote the cardiality of a set, and superscript $\top$ to denote matrix or vector transpose. We use $\{\bs a^j\}$ to denote a collection of all $\bs a^j$'s for $j\in\mJ$, and $(\bs a)_j$ to denote the $j$th element of vector $\bs a$. The concatenation of two vectors $[\bs a; \bs b]$ is equivalent to $[\bs a^\top, \bs b^\top]^\top$.


\section{Related Work} \label{sec:related}

There have been several works on optimization with energy storages and renewable generation, and we briefly review some of them here. In \cite{sg:gonzalex08}, the authors investigated the combined optimization of a wind farm and a pumped storage facility from the perspective of a generation company, using a two-step stochastic optimization approach. The optimization produces optimal bids for the day-ahead spot market, and optimal operation strategies of the facilities. The optimal planning of generation and energy storage capacity was not considered. Zhou {\em et al.} \cite{sg:zhou11ces} proposed a composite energy storage system that contains both high energy density storage and high power density storage. The proposed power converter configuration enables actively distributing demands among different energy storages. Brown {\em et al.} provided an economical analysis of the benefits of having pumped storage in a small island system with abundant renewable energy, and proposed to find the optimal pumped storage capacity through linear programming. In \cite{sg:abbey09}, the authors considered optimizing the rating of energy storage in a wind-diesel isolated grid, and demonstrated that high wind penetration potentially results in significant cost savings in terms of fuel and operating costs.

The main contributions of our work are two fold. First, instead of a single type of energy storage or renewable energy source, we consider a hybrid system with multiple types of energy storage and renewable energy sources, and jointly optimize their capacities and operation. This joint optimization exploits the benefits from each individual element, and therefore is more cost efficient. Second, we propose a distributed optimization framework, so that the capacity design problem becomes scalable when the number of scenarios increases.

\section{System Model} \label{sec:model}

\subsection{Energy storage model}
Assume there is a set $\mS$ of different types of energy storages. We use superscript $\rs \in \mS$ to denote the type of the storage. Each type of energy storage is characterized by a group of parameters. We use $\eta^\rs$ to denote the one-way energy efficiency of energy storage type $\rs$; $\delta^\rs$ to denote the ratio between the rated power and rated energy; $\xi^\rs$ to denote the energy loss ratio per unit time period. The cost of energy storage includes the initial investment cost $c^\rs_\rinv$ and operational/maintenance cost $c^\rs_\rom$. We use $a^\rs$ to denote the amortization factor.

\nomenclature[as]{$\mS$}{Set of different energy storage types}%
\nomenclature[as]{$\rs$}{Types of energy storage}%
\nomenclature[as1]{$\eta^\rs$}{One-way energy efficiency}%
\nomenclature[as1]{$\delta^\rs$}{Rated power/energy ratio}%
\nomenclature[as1]{$\xi^\rs$}{Energy loss ratio per unit time}%
\nomenclature[as1]{$\Ss_{\max}$}{Energy storage capacity}%
\nomenclature[as1]{$P^{\rs,-}_t$}{Charged energy during time period $t$}%
\nomenclature[as1]{$P^{\rs,-}_t$}{Discharged energy time period $t$}%
\nomenclature[b]{$c_\rinv$}{Investment cost per unit storage or generator}%
\nomenclature[b]{$c_\rom$}{Operational/maintenance cost of energy storage or generator}%
\nomenclature[b]{$a$}{Amortization factor per time period}%

Let $\Ss_t$ denote the energy in storage $\rs$ at the beginning of time period $t$, satisfying the following equation:
\begin{equation} \label{eqn:batconsto}
\Ss_{t+1} =
\begin{dcases*}
\Ss_t - \frac{1}{\eta^\rs}P^\rs_t - \xi^\rs \Ss_t   & if $P^\rs_t \geq 0$, \\
\Ss_t - \eta^\rs P^\rs_t - \xi^\rs \Ss_t            & if $P^\rs_t < 0$,
\end{dcases*}
\end{equation}
where positive $P^\rs_t$ denotes discharge from storage $\rs$ during time period $t$, and negative $P^\rs_t$ denotes charge to the storage. Make the following substitution:
\begin{equation} \label{eqn:psub}
P^\rs_t = P^{\rs,+}_t - P^{\rs,-}_t, \quad P^{\rs,+}_t \geq 0,\quad P^{\rs,-}_t\geq 0,
\end{equation}
and we can then rewrite (\ref{eqn:batconsto}) as
\begin{equation} \label{eqn:batconst1}
\Ss_{t+1} = \Ss_t - \frac{1}{\eta^\rs}P^{\rs,+}_t + \eta^\rs P^{\rs,-}_t - \xi^\rs \Ss_t.
\end{equation}

An interpretation of (\ref{eqn:batconst1}) is that the energy stored in a specific energy storage type equals the stored energy at the beginning of the previous time point, minus (plus) the discharge (charge) during the previous time period, minus the energy loss due to the nature of the storage.

The amount of stored energy and the charge/discharge power is constrained by the capacity of the storage, i.e.,
\begin{equation} \label{eqn:batconst2}
0 \leq \Ss_t \leq \Ss_{\max},
\end{equation}
\begin{equation}\label{eqn:batconst3}
0 \leq P^{\rs,+}_t \leq P^{\rs,+}_{\max}, \quad
0 \leq P^{\rs,-}_t \leq P^{\rs,-}_{\max}.
\end{equation}
In this work we use $\delta^\rs$ to denote the ratio between the rated power and the rated storage capacity. Therefore $P^{\rs,+}_{\max} = \eta^\rs \delta^\rs \Ss_{\max}$ and $P^{\rs,-}_{\max} = \delta^\rs \Ss_{\max}$. If the ratio is not fixed, we can introduce another design variable for the rated power, and modify the investment cost so that it depends on both $\Ss_{\max}$ and $\{P^{\rs,+}_{\max}, P^{\rs,-}_{\max}\}$.

The cost of each type of energy storage during time period $t$, denoted by $C^\rs_t$, includes the amortized investment cost and the operational/maintenance cost, i.e.,
\begin{equation}
C^\rs_t = a^\rs c^\rs_\rinv \Ss_{\max}+ c^\rs_\rom(P^{\rs,+}_t, P^{\rs,-}_t).
\end{equation}
In this equation, the operational/maintenance cost depends on the amount of charge and discharge, and any fixed cost can be included as a constant term in this cost function.

\nomenclature[as1]{$C^\rs_t$}{Energy storage cost during time period $t$}%

Note that we made substitution (\ref{eqn:psub}), and therefore constraints (\ref{eqn:batconsto}) and (\ref{eqn:batconst1}) are equivalent if only one element of each pair $\{P^{\rs,+}_t, P^{\rs,-}_t\}$ is non zero for all $\rs, t$. Theorem \ref{thm:equivalent} (in Sec. \ref{subsec:formulation}) guarantees that this condition is satisfied, and therefore the two constraints are indeed equivalent.

\subsection{Generator model}
The generators are classified into traditional diesel generators and renewable generators. For diesel generators, the constraints include the generation capacity and generator ramp constraints. Let $\mH$ denote the set of all diesel generators, and $H^\rh_t$ denote the generation of generator type $\rh \in \mH$ during time period $t$. We then have
\begin{equation} \label{eqn:genconst1}
0 \leq H^\rh_t \leq H^\rh_{\max},
\end{equation}
\begin{equation}\label{eqn:genconst2}
H_\mathrm{ramp}^{\rh,-} \leq H^\rh_{t+1} - H^\rh_t \leq H_\mathrm{ramp}^{\rh,+},
\end{equation}
where $H^\rh_{\max}$ denotes the maximum generation capacity, and $H_\mathrm{ramp}^{\rh,-}$ and $H_\mathrm{ramp}^{\rh,+}$ denote ramp down and ramp up constraints, respectively.
The cost of diesel generators consists of the amortized investment cost and the operational/maintenance cost, denoted by
\begin{equation}
C^\rh_t = a^\rh c^\rh_\rinv H^\rh_{\max}+ c^\rh_\rom(H^\rh_t).
\end{equation}
Usually a second-order quadratic function or piece-wise linear function is used for $c^\rh_\rom(H^\rh_t)$. Any environmental tax can also be included in this cost function.

\nomenclature[at]{$\mH$}{Set of different of diesel generators}%
\nomenclature[at]{$\rh$}{Types of diesel generators}%
\nomenclature[at1]{$H^\rh_{\max}$}{Maximum generation capacity}%
\nomenclature[at11]{$H_\mathrm{ramp}^{\rh,-}$}{Ramp down constraint}%
\nomenclature[at11]{$ H_\mathrm{ramp}^{\rh,+}$}{Ramp up constraint}%
\nomenclature[at1]{$H^\rh_t$}{Diesel generation during time period $t$}%
\nomenclature[at1]{$C^\rs_t$}{Diesel energy cost during time period $t$}%

We employ multiple types of renewable generators, including wind and solar, which are considered as non-dispatchable generations. Let $R^\rr_t$ denote the renewable generation from type $\rr \in \mR$ generator during time period $t$, and $R^\rr_{\max}$ denote the installed capacity. Then the generation can be written as
\begin{equation} \label{eqn:genconst3}
R^\rr_t = r^\rr_t R^\rr_{\max},
\end{equation}
where $r^\rr_t$ is a random variable denoting the renewable generation per unit generation capacity. The cost for renewable energy during time period $t$ is then
\begin{equation}
C^\rr_t = a^\rr c^\rr_\rinv R^\rr_{\max}+ c^\rr_\rom(R^\rr_t).
\end{equation}

\nomenclature[ar]{$\mR$}{Set of different renewable generators}%
\nomenclature[ar]{$\rr$}{Types of renewable generators}%
\nomenclature[ar1]{$R^\rr_{\max}$}{Maximum generation capacity}%
\nomenclature[ar1]{$R^\rr_t$}{Renewable generation during time period $t$}%
\nomenclature[ar1]{$ r^\rr_t$}{Renewable generation per unit generation capacity during time period $t$}%
\nomenclature[ar1]{$C^\rr_t$}{Renewable energy cost during time period $t$}%

In addition to the generator types we discuss here, other types of generators, e.g., hydro and nuclear generators can also be modeled similarly and included in the planning problem.

\subsection{Load balance constraint}
The total generation should equal the total demand in a power grid at all times. Let $G_t$ denote the energy shortage for an isolated grid, or the energy drawn from the main grid for a micro-grid. The total generation and discharge from the energy storages should be equal to the total consumption and charge to the energy storages. We can then write the load balance constraint as follows:
\begin{equation} \label{eqn:lbconst}
D_t = \sum_{\rr \in \mR}R^\rr_t + \sum_{\rh \in \mH}H^\rh_t + \sum_{\rs \in \mS}\left(P^{\rs,+}_t - P^{\rs,-}_t\right) + G_t, \forall t,
\end{equation}
where $D_t$ denotes the demand from users. Note that $G_t$ can be negative, which denotes energy injection to the main grid from a micro-grid, or dumped energy in an isolated grid.

\nomenclature[au]{$D_t$}{Energy demand from users during time period $t$}%
\nomenclature[au]{$G_t$}{Energy shortage or energy drawn from main grid during time period $t$}%

\section{Storage and Renewable Generation Planning} \label{sec:opt}

\subsection{Optimal planning problem} \label{subsec:formulation}
The planning goal is to find the optimal portfolio of different types of energy storage and generators, so that the total cost (including investment and operational/maintenance) is minimized, while most of the needs of the grid can be satisfied. Let $\mT$ denote the planning horizon, and the objective function can then be written as
\begin{equation}
\begin{split}
\begin{aligned}
f(\bs S_{\max}, \bs R_{\max},& \bs H_{\max}) = \\
\sum_{t \in \mT} &\left(\sum_{\rs \in \mS} C^\rs_t + \sum_{\rr \in \mR} C^\rr_t + \sum_{\rh \in \mH} C^\rh_t\right).
\end{aligned}
\end{split}
\end{equation}
\nomenclature[au]{$\mT$}{Set of time periods in planning horizon}%
\nomenclature[c]{$f$}{Objective function for planning problem}%

Due to the intermittency of renewable energy sources, it is possible that in extreme cases, the total local generation will not meet the total demand. We write the grid reliance constraint (for micro-grids) or the energy shortage constraint (for isolated grids) as
\begin{equation} \label{eqn:grd}
G_t^j \leq G_\mathrm{th},
\end{equation}
where $G_\mathrm{th}$ is a threshold which can be a function of current time and demand. There are also constraints on the minimum and maximum capacity for each type of storage and generator, which are denoted as
\begin{equation} \label{eqn:capconst}
\begin{split}
&S_{\min}^{\rs, \rcap} \leq S^\rh_{\max} \leq S_{\max}^{\rs, \rcap}, \forall \rs, \\
&R_{\min}^{\rr, \rcap} \leq R^\rh_{\max} \leq R_{\max}^{\rr, \rcap}, \forall \rr, \\
&H_{\min}^{\rh, \rcap} \leq H^\rh_{\max} \leq H_{\min}^{\rh, \rcap}, \forall \rh.
\end{split}
\end{equation}
We then formulate the optimization problem for energy planning as
\begin{eqnarray}  \label{eqn:opt}
\begin{aligned}
& \mathop{\min}_{\bs S_{\max}, \bs R_{\max}, \bs H_{\max}} & & f(\bs S_{\max}, \bs R_{\max},
\bs H_{\max}) \\
& \mathrm{subject~to~} & & \text{Storage const. } (\ref{eqn:batconst1}) - (\ref{eqn:batconst3}),  \quad \forall \rs, t \\
& & & \text{Generator const. } (\ref{eqn:genconst1}) - (\ref{eqn:genconst3}),  \quad \forall \rr, \rh, t\\
& & & \text{Load balance const. } (\ref{eqn:lbconst}),  \quad \forall t\\
& & & \text{Energy shortage const. } (\ref{eqn:grd}),  \quad \forall t \\
& & & \text{Capacity const. } (\ref{eqn:capconst}).
\end{aligned}
\end{eqnarray}

One problem with this formulation is whether constraints (\ref{eqn:batconsto}) and (\ref{eqn:batconst1}) are equivalent. Based the problem setup, we have the following theorem.
\begin{thm} \label{thm:equivalent}
In a cost minimization context, given an increasing positive operational cost function for charging and discharging, by making the substitution (\ref{eqn:psub}), we have that $P^{\rs,+}_t P^{\rs,-}_t =0$ for all $t$, i.e., only one of $P^{\rs,+}_t$ and $P^{\rs,-}_t$ can be non-zero for any given time period $t$.
\end{thm}
\begin{proof}
See Appendix \ref{proof:thm1}.
\end{proof}

Additional costs and constraints can also be easily included in this formulation. For example, an environmental tax for traditional diesel generators, and government incentives for renewable generations can be included in the corresponding cost functions. The maximum allowed diesel generation capacity specified by certain energy policies can be included in the generator constraints.

{\em Remark 1:} The problem formulation can be slightly modified into a chance-constrained problem. Instead of the deterministic constraint (\ref{eqn:grd}), we can use the following probabilistic constraint:
\begin{equation} \label{eqn:grpr}
\Pr(G_t \geq G_\mathrm{th}) \leq \alpha,
\end{equation}
where $\alpha \in [0,1]$ is the maximal energy shortage probability allowed. Constraint (\ref{eqn:grpr}) means that local generators and storages have a probability less than or equal to $\alpha$ to be short of energy greater than $G_\mathrm{th}$. In this case, using the results from \cite{sg:calafiore06}, \cite{book:stoopt1}, the probabilistic constraint can be approximated by a set of deterministic constraints, sampled from the probability distribution of the random parameters from the probabilistic constraint. To be more specific, let each scenario be a random realization of load, renewable generation, and initial conditions of the energy storages. The number of required scenarios $J = \mathrm{card}(\mJ)$ is determined by the number of design parameters and the probability measure. Let $N$ denote the number of design parameters. If the number of scenarios $J$ is no less than $\lceil 2N\alpha^{-1}\ln(2\alpha^{-1})+2\alpha^{-1}\ln(\epsilon^{-1})+2N \rceil$, then the solution to the scenario approximation problem has a probability at least $1-\epsilon$ to satisfy the original chance constrained problem. The problem formulation and method of solving the problem are very similar to (\ref{eqn:opt}). We will point out the difference in {\em Remark 2}.

\nomenclature[d]{$G_\mathrm{th}$}{Energy shortage threshold}%
\nomenclature[d]{$\alpha$}{Maximal energy shortage probability allowed}%

\subsection{Formulation of consensus problem}

The renewable generation and user loads in (\ref{eqn:opt}) are all random. In practice, historical data is used in the problem formulation. With a large number of realizations of the random parameters from historic data, the problem becomes increasingly difficult to solve due to the increase of dimensionality. In the rest of this section, we will reformulate the original problem (\ref{eqn:opt}) as a consensus problem, which can be solved in a distributed manner.

We divide the entire planning horizon $\mT$ into sub planning horizons $\mT^j$, which we call scenarios for simplicity. Let $\mJ$ denote the set of all horizons, and we have that $\mT = \cup_{j \in \mJ}\mT^j$. For convenience, we assume $\mT^j$'s are arranged in the order of time. Let $\x_\rd^j = [\bs S^{j}_{\max}; \bs R^j_{\max}; \bs H^j_{\max}]$ denote the design parameters for the $j$th scenario, $\z_\rd = [\bs S_{\max}; \bs R_{\max}; \bs H_{\max}]$ denote the global design parameters, $\mC^j$ denote the feasible set for the $j$th scenario with $\mC = \cap_{j\in\mJ}\mC^j$. In practice, the energy in storages at the beginning of each time period is not random, but rather depends on the energy from the previous time period. Assuming the energy stored at the beginning of a scenario should be equal to the energy stored at the end of the previous scenario, we need additional constraints to ensure this condition is satisfied. Let $\bs S^j_0$ denote the energy storage at the beginning of the $j$th scenario, and $\bs S^j_T$ denote the energy storage at the end of the $j$th scenario. In \cite{sg:yang13storagecv}, we followed the approach in \cite{sg:brown08}, and imposed an additional assumption that the energy in each energy storage at the end of the optimization horizon should be equal to that at the beginning of the optimization horizon, i.e., $\bs S^j_0 = \bs S^{j}_T, j\in\mJ$. However, this assumption makes the solution suboptimal. In this work, we eliminate this assumption and add additional consensus constraints across scenarios.

\nomenclature[e1]{$\mB$}{Group index mapping for boundary conditions}%
\nomenclature[e1]{$B$}{Element-wise index mapping for boundary conditions}%

Let $\x_\rb^j = [\bs S^j_0; \bs S^{j}_T]$ denote the boundary parameters for the $j$th scenario, and $\z_\rd$ denote the global boundary parameters. Let $\mB(j)$ denote the mapping for the indices of the boundary conditions for the $j$th scenario, and thus $\bs z_{\rb,\mB(j)}$ denotes the corresponding boundary parameters for the $j$th scenario. We also use the scalar function $B(j, i)$ to denote element-wise index mapping, i.e., $(\x_\rb^j)_i$ corresponds to $(\z_\rb)_{B(j,i)}$. The constraints $\bs S^j_0 = \bs S^{j-1}_T, j\in\mJ$ can then be written as $\x_\rb^j = \z_{\rb,\mB(j)}, j\in\mJ$.
Using the notations $\bs x^j = [\x^j_\rd; \x^j_\rb]$, $\bs z = [\z_\rd; \z_\rb]$, and $\tbz^j = [\z_\rd; \z_{\rb, \mB(j)}]$, we then formulate the original optimization problem (\ref{eqn:opt}) as follows:
\begin{eqnarray}  \label{eqn:optsa}
\begin{aligned}
& \mathop{\min}_{\x^j \in \mC^j} & & \sum_{j\in\mJ} f^j(\x^j) \\
& \mathrm{subject~to~} & & \x^j = \tbz^j, \quad j\in \mJ. \\
\end{aligned}
\end{eqnarray}
The solution $\bs z$ from solving (\ref{eqn:optsa}) will satisfy that $\z \in\mC$.

{\em Remark 2:} If the probabilistic constraint is considered, and scenario approximation approach is used, the formulation have to be slightly revised. According to \cite{sg:calafiore06}, the random samples for each scenario has to be generated from independent identical distributions. Note that the starting energy stored in the storages also has to be drawn from certain probability distributions. The consensus formulation for the energy storage boundary conditions can then be removed. The number of generated scenarios has to be greater than or equal to the minimum number described in {\em Remark 1}.

\nomenclature[e]{$\mC^j$}{Feasible set for the $j$th scenario}%
\nomenclature[e]{$\x_\rd^j$}{Design parameters for the $j$th scenario}%
\nomenclature[e]{$\z_\rd$}{Global design parameters}%
\nomenclature[e]{$\x_\rb^j$}{Boundary parameters for the $j$th scenario}%
\nomenclature[e]{$\z_\rb$}{Global boundary parameters}%

\nomenclature[e2]{$\rho$}{Dual variable update step-size}%
\nomenclature[e2]{$\tau$, $\mu$}{Parameters for adaptive dual variable step-size}%

\subsection{Distributed Optimization}
The challenge in solving (\ref{eqn:optsa}) is that as the number of scenarios increases, the problem becomes increasingly difficult due to high time complexity. We propose to solve the problem in a distributed manner based on the alternating direction method of multipliers (ADMM) \cite{opt:admm10}, which mitigates the time complexity issue and makes the problem scalable.

To enforce the equality (consensus) constraint in (\ref{eqn:optsa}), an additional quadratic term is added to the original Lagrangian, forming the augmented Lagrangian which can be written as
\begin{equation}
\begin{split}
L_\rho\big(&\{\x^j\}, \z, \{\v^j\}\big)= \\
& \sum_{j \in \mJ} \left( f^j(\x^j) + \v^{j\top} (\x^j - \z) + \frac{\rho}{2} \| \x^j - \z \|_2^2\right),
\end{split}
\end{equation}
where $\{\v^j\}$ denote the dual variables, and $\rho$ is a pre-defined parameter which is the dual variable update step size. The quadratic term penalizes the difference between the local variables $\{\x^j\}$ and corresponding entries of the global variable $\z$, denoted by $\tbz^j$.

The ADMM algorithm iterates among the following steps, with subscript $k$ denoting the iteration number.

\subsubsection{x-minimization step}
For each $j \in \mJ$, the following local minimization problems are solved in parallel:
\begin{equation} \label{eqn:xupdate}
\begin{aligned}
\x_{k+1}^j &= \\
&\mathop{\mathrm{argmin}}_{\x^j \in \mC^j}  f(\x^j) + \v^{j\top} (\x^j - \tbz^j_k) + \frac{\rho}{2} \| \x^j - \tbz^j_k \|_2^2.
\end{aligned}
\end{equation}

\subsubsection{z-minimization step}
\begin{equation} \label{eqn:zupdate}
\begin{aligned}
\z&_{k+1} = \\
&\mathop{\mathrm{argmin}}_{\z_k \in \mC}  \sum_{j \in \mJ} \left(\v^{j\top} (\x^j_{k+1} - \tbz^j_k) + \frac{\rho}{2} \| \x^j_{k+1} - \tbz^j_k \|_2^2\right).
\end{aligned} \raisetag{4\baselineskip}
\end{equation}
To solve for the {\em z-minimization} step, we consider $\z_\rb$ and $\z_\rd$ separately. Decompose $\v^j = [\v^j_\rd; \v^j_\rb]$, and we then rewrite (\ref{eqn:zupdate}) as

\begin{equation} \label{eqn:zupdates}
\begin{aligned}
\z&_{k+1} = \\
&\mathop{\mathrm{argmin}}_{\z}  \sum_{j \in \mJ} \left(\v_\rd^{j\top} (\x^j_{\rd,k+1} - \z_{\rd}) + \frac{\rho}{2} \| \x^j_{\rd,k+1} - \z_{\rd} \|_2^2\right. + \\
&\qquad \left. \v_\rb^{j\top} (\x^j_{\rb,k+1} - \z_{b,\mB(j)}) + \frac{\rho}{2} \| \x^j_{\rb,k+1} - \z_{b,\mB(j)} \|_2^2\right).
\end{aligned} \raisetag{4\baselineskip}
\end{equation}

Solving (\ref{eqn:zupdates}), we obtain that
\begin{equation} \label{eqn:zdupdate}
\z_{\rd,k+1} = \frac{1}{J} \sum_{j\in\mJ} \left(\x^j_{k+1} + \frac{1}{\rho} \v^j_k\right),
\end{equation}
\begin{equation} \label{eqn:zbupdate}
(\z_{\rb,k+1})_g = \frac{\sum_{B(j,i) = g}\left((\x^j_{\rb,k+1})_i + (1/\rho)(\v^j_{\rb})_i\right)}{\sum_{B(j,i) = g}1}.
\end{equation}

When the algorithm converges, the resulting $\bs z^j$ has to satisfy the constraints of each sub-problem, i.e., $\tbz^j \in \mC^j$. Therefore we have that $\z \in \mC \ = \cap_{j\in\mJ}C^j$.


\subsubsection{Dual-variable update}
For each $j \in \mJ$, the dual variables are updated in parallel:
\begin{equation} \label{eqn:vupdate}
\v_{k+1}^j = \v_k^j + \rho \left(\x^j_{k+1} - \tbz_{k+1}\right).
\end{equation}

Since (\ref{eqn:xupdate}) and (\ref{eqn:vupdate}) can be parallelized, the problem is scalable as the number of scenarios increases. The convergence of this approach is guaranteed, as proved in \cite{opt:admm10}. For faster convergence, we use an adaptive dual update stepsize $\rho$. The primal residual $\r^\rp_k$ and dual residual $\r^\rd_k$ are defined as
\begin{equation}
\r^\rp_{k+1} = \frac{1}{J}\sum_{j \in\mJ} \left(\x^j - \tbz^j\right),
\end{equation}
\begin{equation}
\r^\rd_{k+1} = \rho(\z_{k+1}-\z_k).
\end{equation}
As larger $\rho$ penalizes more on the primal residual, and smaller $\rho$ penalizes on the dual residual, the parameter $\rho$ is updated following the rule below:
\begin{equation}
\rho_{k+1} = \begin{dcases*}
\tau \rho_k & if $\|\r^\rp_k\| > \mu \|\r^\rd_k\|$, \\
\rho_k/\tau & if $\|\r^\rd_k\| > \mu \|\r^\rp_k\|$, \\
\rho_k & otherwise,
\end{dcases*}
\end{equation}
where $\tau>1$, $\mu>1$. The algorithm converges when both the primal and dual residual are less than a certain threshold.

\section{Numerical Examples} \label{sec:exp}

In this section, we provide a series of numerical examples using real data from online databases, to showcase how the proposed framework can help in making decisions on renewable generation and energy storage planning.

\begin{table*}
\begin{center}
\caption {Parameters for energy storage and generators.}
\begin{tabular} {|c| l | c  c  c | c  c  c | } \hline
\multirow{2}{*}{Category} & \multirow{2}{*}{Type} & Round-trip & Full charge time & Energy loss & Investment Cost & Life span &\multirow{2}{*}{O/M cost} \\
  &                      & Efficiency  & (hours)         & Ratio$^{\dagger}$         & (M\$/MWh)       & (years) &     \\ \hline
\multirow{3}{*}{Energy storage}&
  $S^1$: Flywheel         & 0.92        & 0.25            & 0.05          & 7.800           & 20      & linear          \\
& $S^2$: Li-ion battery   & 0.88        & 4.00            & 0.01          & 1.700           & 15      & linear          \\
& $S^3$: Pumped storage   & 0.80        & 10.00           & 0.00          & 0.450           & 50      & linear          \\ \hline
\multirow{2}{*}{Renewable generator}&
  $R^1$: Solar panel      & --          & --              & --            & 5.284           & 30      & linear          \\
& $R^2$: Wind turbine     & --          & --              & --            & 2.414           & 20      & linear          \\ \hline
Traditional generator &
  $H^1$: Diesel generator & --          & --              & --            & 0.400           & 5       & quadratic       \\ \hline
\end{tabular} \label{t:srpara}
\\
\begin{flushleft}\footnotesize{$^\dagger$ This parameter was not provided in \cite{sg:storageopt10}.} \end{flushleft}
\end{center}
\end{table*}

\subsection{Data and parameters} \label{subsec:para}

\subsubsection{Renewable generation data}

We consider two types of renewable generation, wind and solar, and simulate renewable generation data from National Solar Radiation Data Base (NSRDB) from NREL \cite{sg:solardata}, \cite{sg:solarradman}. The database provides hourly solar radiation data as well as wind speed data.
The solar generation is calculated using the hourly ``modeled direct normal" radiation measurements and the hourly mean zenith angle. Denote the normal radiation measurement as $R^\rs_{\mathrm{n}}$, the longitude of the target location as $\theta_{\mathrm{l}}$, and the solar panel tilt angle as $\theta_{\mathrm{t}}$. The power received on panel $R^\rs_{\mathrm{p}}$ is then calculated by
\begin{equation}
R_\rs = \min \left( \mu^\rs R^\rs_{\mathrm{n}}\cos(\theta_{\mathrm{t}} - \theta_{\mathrm{l}}), R^\rs_\rr \right),
\end{equation}
where $\mu_\rs$ is the panel efficiency and $R^\rs_\rr$ is the rated power output. The optimal tilt angle for solar panel is determined according to \cite{sg:solartilt}.

The wind power output is calculated using the following equation:
\begin{equation}
R_\rw = \begin{dcases*}
\frac{1}{2} \mu_\rr\rho_{\mathrm{air}} V^3 \frac{\pi d^2}{4} & if $V_{\mathrm{in}} \leq V \leq V_{\mathrm{rated}}$, \\
\frac{1}{2} \mu_\rr\rho_{\mathrm{air}} V_{\mathrm{rated}}^3 \frac{\pi d^2}{4} & if $V_{\mathrm{rated}} \leq V \leq V_{\mathrm{out}}$, \\
0 & otherwise,
\end{dcases*}
\end{equation}
where $\rho_{\mathrm{air}}$ is the density of air, $d$ is the diameter of the wind turbine, and $\mu_\rr$ is the wind turbine efficiency. The generation output is zero when the wind speed is lower than the cut-in speed $V_\mathrm{in}$ or higher than the cut-out speed $V_\mathrm{out}$.

Both the solar and wind generations are normalized by the rated power outputs.
The costs for renewable generators are obtained from \cite{sg:renewableAsses07} and included in Table \ref{t:srpara}.

\subsubsection{Load data}

We use the ERCOT hourly load data from \cite{sg:ercotload}. The data is normalized by the average hourly demand.

\subsubsection{Energy storage parameters}

We select three types of energy storage as prototypes for our simulations, including flywheel storage, Li-ion battery, and pumped storage. The corresponding parameters are determined based on \cite{sg:storageopt10} and included in Table \ref{t:srpara}.

\subsubsection{Definitions}
To quantify $G_\mathrm{th}$ and the maximum capacity of diesel generator $H^{\rcap}_{\max}$, we define two quantities: the shortfall-to-demand ratio, $r_\mathrm{SD}$, and the diesel generation capacity ratio, $r_\mathrm{DC}$. We define the threshold $G_\mathrm{th}$ at time $t$ as
$
G_{\mathrm{th}} = r_\mathrm{SD}  D_t,
$
and therefore the shortfall-to-demand ratio is the ratio between the threshold $G_\mathrm{th}$ and the current demand. The  maximum diesel generator capacity $H^\rcap_{\max}$ is determined by
$
H^\rcap_{\max} = r_\mathrm{DC} \max(D_t),
$
and therefore the diesel generation capacity ratio is the ratio between the maximum diesel generator capacity and the peak demand.

\nomenclature[f]{$r_\mathrm{DC}$}{Maximum diesel generation capacity ratio}%
\nomenclature[f]{$r_\mathrm{SD}$}{Shortfall-to-demand ratio}%

\subsection{Result of storage and generation planning} \label{sec:resultsan}

In this subsection we perform a case study using the setup described in Sec.\ \ref{subsec:para}. The solar panel tilt angles are adjusted twice a year according to \cite{sg:solartilt}. Solar panel efficiency is set to be $20\%$, with rated power output to be $150$W/m$^2$. Wind turbine cut-in and cut-out wind speeds are set to be $3 \mathrm{m}/\rs$ and $20 \mathrm{m}/\rs$, with rated power output achieved at the wind speed of $10 \mathrm{m}/\rs$. The wind turbine efficiency is set to be $50\%$. In the following simulations, we consider the average hourly load to be unit megawatt (1MW) for illustrative purposes, while a micro-grid is usually on the scale of 5-10MW. We use the ERCOT hourly load data from years 2008-2010 to generate the load data, and the NSRDB data to generate the renewable generation data. The maximum hourly load is $1.8050$ MWh, and the minimum hourly load is $0.5557$ MWh. The average hourly generations from per unit MW of wind turbines and solar panels are $0.1217$ MWh and $0.2366$ MWh, respectively. The normalized data (in heatmap) and corresponding box plots are shown in Fig.\ \ref{fig:data}.

\begin{figure}
  \centering
   \subfloat{\includegraphics[width=0.48\textwidth]{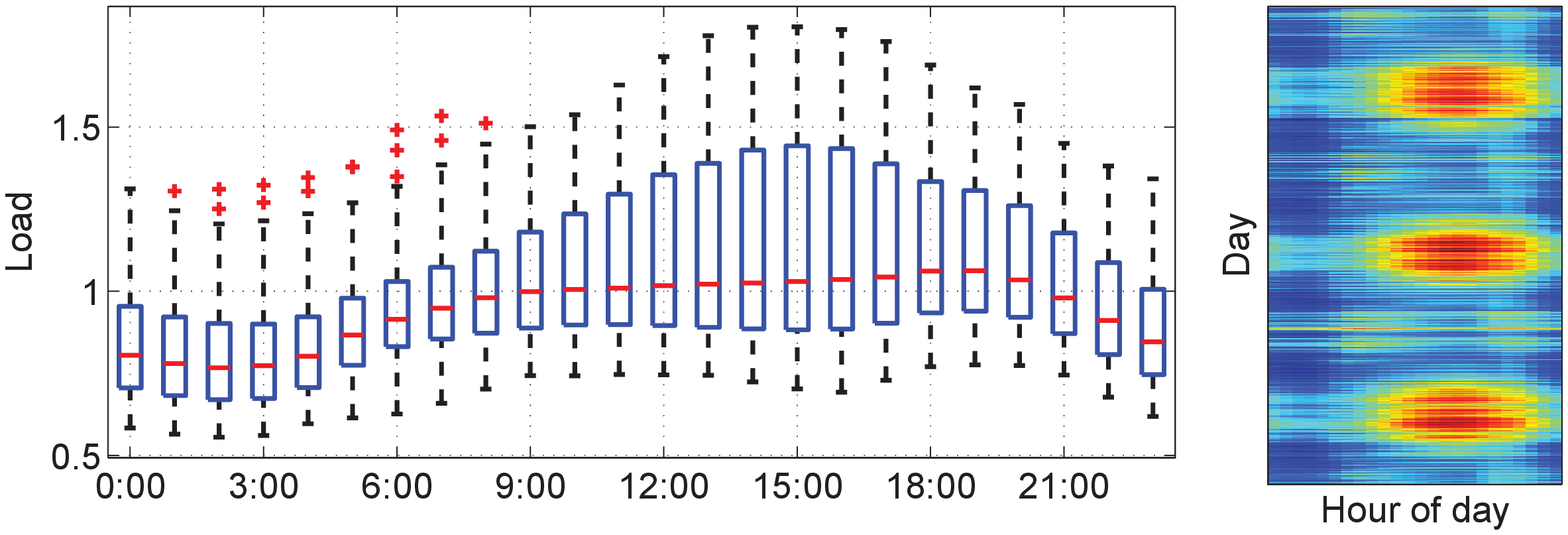}} \\
   \subfloat{\includegraphics[width=0.48\textwidth]{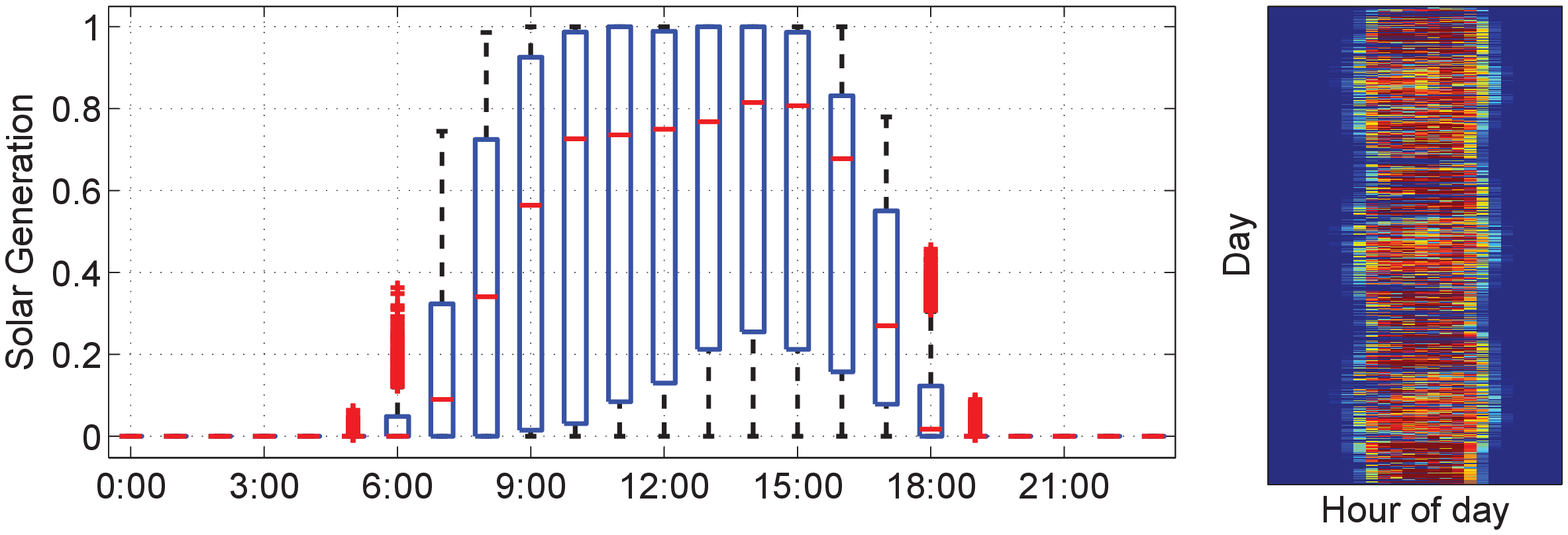}} \\
   \subfloat{\includegraphics[width=0.48\textwidth]{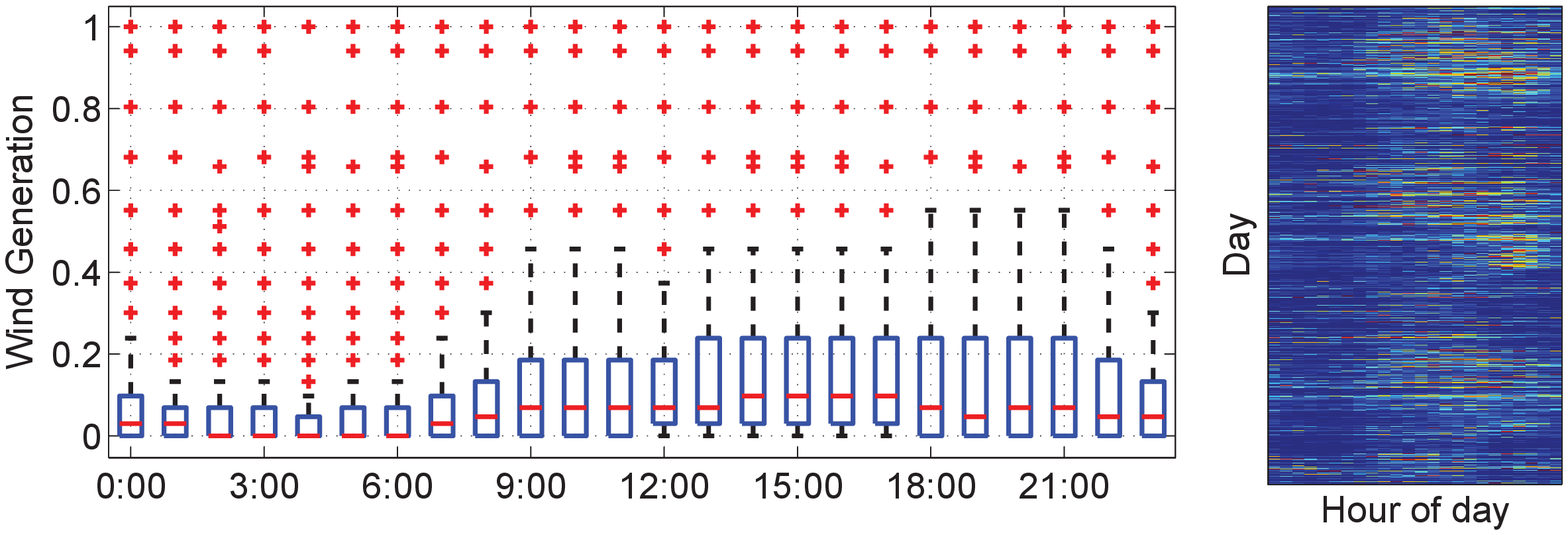}}
   \caption{Plots for normalized load, solar generation, and wind generation data. The plots on the left show box plots, and the plots on the right show raw data.}
    \label{fig:data}
\end{figure}

We set $r_\mathrm{DC} = 0.50$ and $r_\mathrm{SD} = 0.05$, meaning that the maximum allowed diesel generation capacity is half of the peak load, and the maximum energy shortfall to demand from local generators is $5\%$. The CVX toolbox \cite{cvx_n} is used to solve for the x-updates. The optimization results are available in Table \ref{t:planningresult}. All costs are for a thee-year design horizon. The investment costs are amortized for three years.

We notice that although the planned diesel generation capacity is high (equivalent to the upper limit, which is half of the peak load), the actual average generation from these generators are relatively low. Overall, only $18.22\%$ of the consumed energy is from diesel generators, i.e., renewable energy constitutes $81.78\%$ of the total consumed energy\footnote{We assume the renewable generators are generating as much energy as they can, and thus there is excess generation on certain occasions. Only the actual energy consumed by the end users is counted here.}.

\begin{table}
\begin{center}
\caption {Results for $r_\mathrm{DC} = 0.50$, $r_\mathrm{SD} = 0.05$ over a three-year design horizon.}
\begin{tabular} {| c | c  c   c  c |} \hline
\multirow{2}{*}{Type}    & Planned capacity      & Investment cost       & O./M. cost     & Total cost      \\
                         & (MWh)                 &  (M\$)                & (M\$)          & (M\$)           \\ \hline
$S^1$   & 0.5855                & 0.6851              & 0.0019               & 0.6870               \\
$S^2$   & 3.2595                & 1.1082              & 0.0274               & 1.1356               \\
$S^3$   & 4.3051                & 0.1162              & 0.0390               & 0.1553               \\ \hline
$R^1$   & 2.4985                & 1.3202              & 0.1554               & 1.4755               \\
$R^2$   & 5.5192                & 1.9985              & 0.0883               & 2.0868               \\ \hline
$H^1$   & 0.9025                & 0.2166              & 0.5336               & 0.7502               \\ \hline
Total   & --                    & 5.4448              & 0.8456               & 6.2904               \\ \hline
\end{tabular} \label{t:planningresult}
\end{center}
\end{table}

\begin{figure}
  \centering
   \subfloat{\includegraphics[width=0.48\textwidth]{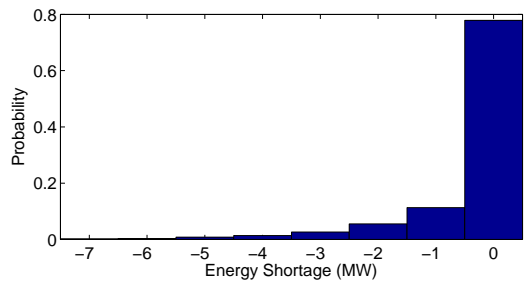}}
   \caption{Hourly energy shortage distribution. Negative value denotes that generation is greater than demand.}
    \label{fig:esd}
\end{figure}

The hourly energy shortage distribution is plotted in Fig.\ \ref{fig:esd}. For most of the time, the energy shortage $G_t$ is close to zero, meaning that the total generation (including storage discharge) is close to the total demand (including storage charge). There are also occasions when there is excess generation, which is either dumped in isolated grids or injected to the main grid in grid-connected micro-grids.

\subsection{Result with different diesel generation capacities}

The maximum allowed diesel generation capacity affects the planning for renewable generation and energy storages. Intuitively, the lower the allowed diesel generation capacity, the more renewable generators and energy storages are required to ensure an uninterrupted energy supply. In this section, we consider different diesel capacity ratios $r_\mathrm{DC}$ while keeping other parameters fixed. The diesel capacity ratio $r_\mathrm{DC}$ is changed from $0$ to $1$, with increments of $0.1$. The resulting total cost and average renewable generation percentage over the three planning years are illustrated in Fig.\ \ref{fig:dcr}.

\begin{figure}
  \centering
   \subfloat{\includegraphics[width=0.48\textwidth]{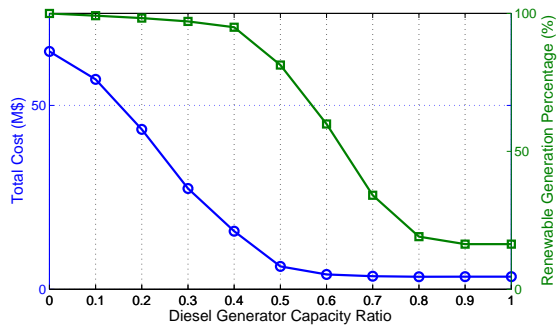}}
   \caption{Total cost and percentage of renewable generation as a function of diesel generation capacity ratio.}
    \label{fig:dcr}
\end{figure}

It can be observed that when $r_\mathrm{DC}$ is smaller then $0.5$ and decreases, the total cost increases rapidly, whereas the percentage of renewable energy do not increase much. This potentially provides guidelines for determining the capacity of diesel generators in an micro-grid, when balancing financial cost and environmental cost.

\begin{table}
\begin{center}
\caption {Geographic locations and climates.}
\begin{tabular} {| c | c | c |} \hline
City                    & GPS Coordinates       & Climate type  \\ \hline
San Antonio             & 29$^\circ$25'N 98$^\circ$30'W       & humid subtropical / hot semi-arid  \\ \hline
St. Louis               & 38$^\circ$37'N 90$^\circ$11'W       & humid continental / subtropical   \\ \hline
San Francisco           & 37$^\circ$47'N 122$^\circ$25'W       & cool-summer Mediterranean \\ \hline
\end{tabular} \label{t:location}
\end{center}
\end{table}

\begin{table} [!h]
\begin{center}
\caption {Comparison of optimization results for different geographic locations.}
\begin{tabular} {| c | c  c   c |} \hline
                    & San Antonio  & St. Louis     & San Francisco      \\ \hline
$S^1$ (MWh)         & 0.5855       & 0.1078        & 0.0013       \\
$S^2$ (MWh)         & 3.2595       & 3.1140        & 3.8601       \\
$S^3$ (MWh)         & 4.3051       & 4.5006        & 5.3880       \\ \hline
$R^1$ (MWh)         & 2.4985       & 2.3587        & 5.2976       \\
$R^2$ (MWh)         & 5.5192       & 4.8909        & 3.2458       \\ \hline
$H^1$ (MWh)         & 0.9025       & 0.9020        & 0.9025       \\ \hline
Total cost (M\$)    & 6.2904       & 5.5331        & 6.5856       \\
Renewable (\%)      & 81.78        & 75.32         & 83.06        \\ \hline
\end{tabular} \label{t:geocomp}
\end{center}
\end{table}

\subsection{Results with data from different geographic locations}

The geographic location has a significant effect on the availability of different renewable energy sources. We consider the case of three cities in the United States, described in Table \ref{t:location}, using a similar setup as described in Sec.\ \ref{sec:resultsan}. The hourly data is not available for all the locations, and therefore we use the same hourly load data for all three regions.
It can be observed from Table \ref{t:geocomp} that the availability of different resources significantly affects the planned capacity of different renewable generators. Since San Francisco has more wind resources, only a relatively small capacity of wind generator is needed to satisfy the energy demands. For similar reasons, San Antonio and St. Louis require less solar panel installation. Other characteristics, e.g., the variability of renewable energy sources, also affect the planned capacity for different types of energy storages.

{\em Remark 3: } In this paper we consider relatively isolated grids. To make better use of renewable generation, it would be more efficient to consider an interconnected network of multiple micro-grids, with inter-grid energy transmission. In this way, less renewable generation and energy storage capacity would be needed to satisfy the energy demands.
The optimization problem will more complicated, and new issues, e.g., long-distance power transmission, need to be considered.


\section{Conclusions} \label{sec:conclusion}

We considered the problem of jointly optimizing multiple energy storage, renewable generator, and diesel generator capacities, in the context of an isolated grid, or a micro-grid with a small carbon footprint. The joint optimization exploits the different characteristics of multiple energy storage types, as well as the availability of different sources of renewable energy. To mitigate the large dimensionality of the optimization problem due to the use of large volumes of historical data, we formulated the original optimization problem as a consensus problem, which can be solved in a parallel distributed manner. We provided a series of numerical examples to illustrate how the proposed framework can be used for planning purposes in practice. To be more specific, we considered scenarios with different maximum diesel generation capacities, and also compared the different planning results in different geographic regions. The proposed work will be helpful in designing renewable generation and energy storage systems for future decentralized power grids with large renewable penetration, and help policy makers make decisions related to renewable energy and sustainability.

In our future work we will consider the problem of optimally operating a given hybrid energy storage and generation system, taking into account the stochastic nature of demand and renewable generation. We will also study the case of an interconnected network of multiple micro-grids, each with local energy generation and inter-grid energy transmission, which potentially makes better use of renewable energy.

\appendix
\subsection{Proof of Theorem 1} \label{proof:thm1}
\begin{proof}
Let $\{P^{\rs,+}_t, P^{\rs,-}_t\}$ and $\{\hat P^{\rs,+}_t, \hat P^{\rs,-}_t\}$ denote two pairs of charge/discharge satisfying
\begin{eqnarray}
\begin{aligned}
&P^{\rs,+}_t - P^{\rs,-}_t = P^\rs_t,\quad  P^{\rs,+}_t P^{\rs,-}_t = 0,\\
&\hat P^{\rs,+}_t - \hat P^{\rs,-}_t = P^\rs_t, \quad \hat P^{\rs,+}_t \hat P^{\rs,-}_t > 0.
\end{aligned}
\end{eqnarray}
We then have
\begin{equation}
\begin{aligned}
(P^{\rs,+}_t + P^{\rs,-}_t)^2 &= (P^{\rs,+}_t - P^{\rs,-}_t)^2 + 4P^{\rs,+}_t P^{\rs,-}_t\\
&= (P^{\rs,+}_t - P^{\rs,-}_t)^2
= (\hat P^{\rs,+}_t - \hat P^{\rs,-}_t)^2 \\
&< (\hat P^{\rs,+}_t - \hat P^{\rs,-}_t)^2 + 4\hat P^{\rs,+}_t \hat P^{\rs,-}_t  \\
&= (\hat P^{\rs,+}_t + \hat P^{\rs,-}_t)^2.
\end{aligned}
\end{equation}
Since $P^{\rs,+}_t - P^{\rs,-}_t = \hat P^{\rs,+}_t - \hat P^{\rs,-}_t$, we then have $P^{\rs,+}_t<P^{\rs,+}_t$, and $P^{\rs,-}_t<P^{\rs,-}_t$. Because the operational cost is an increasing function of $\{P^{\rs,+}_t, P^{\rs,-}_t\}$, we obtain that
\begin{equation}
c_\rom(P^{\rs,+}_t , P^{\rs,-}_t) < c_\rom(\hat P^{\rs,+}_t , \hat P^{\rs,-}_t).
\end{equation}
Therefore the optimal pair $\{P^{\rs,+}_t , P^{\rs,-}_t\}$ must satisfy that $P^{\rs,+}_t P^{\rs,-}_t = 0$, i.e., only one of $P^{\rs,+}_t , P^{\rs,-}_t$ can be non-zero.
\end{proof}

\bibliographystyle{IEEEtran}

\begin{thebibliography}{10}
\providecommand{\url}[1]{#1}
\csname url@samestyle\endcsname
\providecommand{\newblock}{\relax}
\providecommand{\bibinfo}[2]{#2}
\providecommand{\BIBentrySTDinterwordspacing}{\spaceskip=0pt\relax}
\providecommand{\BIBentryALTinterwordstretchfactor}{4}
\providecommand{\BIBentryALTinterwordspacing}{\spaceskip=\fontdimen2\font plus
\BIBentryALTinterwordstretchfactor\fontdimen3\font minus
  \fontdimen4\font\relax}
\providecommand{\BIBforeignlanguage}[2]{{%
\expandafter\ifx\csname l@#1\endcsname\relax
\typeout{** WARNING: IEEEtran.bst: No hyphenation pattern has been}%
\typeout{** loaded for the language `#1'. Using the pattern for}%
\typeout{** the default language instead.}%
\else
\language=\csname l@#1\endcsname
\fi
#2}}
\providecommand{\BIBdecl}{\relax}
\BIBdecl

\bibitem{re:renewable}
``Renewables: energy you can count on,'' Union of Concerned Scientists, Tech.
  Rep., Apr. 2013.

\bibitem{sg:collier10}
S.~Collier, ``Ten steps to a smarter grid,'' \emph{{IEEE} Ind. Appl. Mag.},
  vol.~16, no.~2, pp. 62--68, 2010.

\bibitem{sg:Turner99}
J.~A. Turner, ``A realizable renewable energy future,'' \emph{Science}, vol.
  285, no. 5428, pp. 687--689, 1999.

\bibitem{re:nrel12}
``Exploration of high-penetration renewable electricity futures,'' National
  Renewable Energy Laboratory, Tech. Rep., 2012.

\bibitem{sg:wiedmann08}
T.~Wiedmann and J.~Minx, \emph{A Definition of 'Carbon Footprint'}.\hskip 1em
  plus 0.5em minus 0.4em\relax Hauppauge NY, USA.: Nova Science Publishers,
  2008.

\bibitem{sg:carrasco06}
J.~Carrasco, L.~Franquelo, J.~Bialasiewicz, E.~Galvan, R.~Guisado, M.~Prats,
  J.~Leon, and N.~Moreno-Alfonso, ``Power-electronic systems for the grid
  integration of renewable energy sources: A survey,'' \emph{{IEEE} Trans. Ind.
  Electron.}, vol.~53, no.~4, pp. 1002--1016, 2006.

\bibitem{sg:Ibrahim08}
\BIBentryALTinterwordspacing
H.~Ibrahim, A.~Ilinca, and J.~Perron, ``Energy storage systems --
  characteristics and comparisons,'' \emph{Renewable and Sustainable Energy
  Reviews}, vol.~12, no.~5, pp. 1221 -- 1250, 2008. [Online]. Available:
  \url{http://www.sciencedirect.com/science/article/pii/S1364032107000238}
\BIBentrySTDinterwordspacing

\bibitem{sg:gonzalex08}
J.~Garcia-Gonzalez, R.~de~la Muela, L.~Santos, and A.~Gonzalez, ``Stochastic
  joint optimization of wind generation and pumped-storage units in an
  electricity market,'' \emph{{IEEE} Trans. Power Syst.}, vol.~23, no.~2, pp.
  460--468, 2008.

\bibitem{sg:trong10}
T.~D. Nguyen, K.-J. Tseng, S.~Zhang, and T.~D. Nguyen, ``On the modeling and
  control of a novel flywheel energy storage system,'' in \emph{Industrial
  Electronics (ISIE), 2010 IEEE International Symposium on}, 2010, pp.
  1395--1401.

\bibitem{sg:zhou11ces}
H.~Zhou, T.~Bhattacharya, D.~Tran, T.~Siew, and A.~Khambadkone, ``Composite
  energy storage system involving battery and ultracapacitor with dynamic
  energy management in microgrid applications,'' \emph{Power Electronics, IEEE
  Transactions on}, vol.~26, no.~3, pp. 923--930, 2011.

\bibitem{sg:Chalk06}
S.~G. Chalk and J.~F. Miller, ``Key challenges and recent progress in
  batteries, fuel cells, and hydrogen storage for clean energy systems,''
  \emph{Journal of Power Sources}, vol. 159, no.~1, pp. 73 -- 80, 2006.

\bibitem{sg:barton04}
J.~Barton and D.~Infield, ``Energy storage and its use with intermittent
  renewable energy,'' \emph{Energy Conversion, IEEE Transactions on}, vol.~19,
  no.~2, pp. 441--448, 2004.

\bibitem{sg:vosburgh78}
K.~G. Vosburgh, ``Compressed air energy storage,'' \emph{Journal of Energy},
  vol.~2, no.~2, pp. 106--112, 1978.

\bibitem{sg:abbey07}
C.~Abbey and G.~Joos, ``Supercapacitor energy storage for wind energy
  applications,'' \emph{Industry Applications, IEEE Transactions on}, vol.~43,
  no.~3, pp. 769--776, 2007.

\bibitem{sg:brown08}
P.~Brown, J.~Peas~Lopes, and M.~Matos, ``Optimization of pumped storage
  capacity in an isolated power system with large renewable penetration,''
  \emph{Power Systems, IEEE Transactions on}, vol.~23, no.~2, pp. 523--531,
  2008.

\bibitem{sg:abbey09}
C.~Abbey and G.~Joos, ``A stochastic optimization approach to rating of energy
  storage systems in wind-diesel isolated grids,'' \emph{Power Systems, IEEE
  Transactions on}, vol.~24, no.~1, pp. 418--426, 2009.

\bibitem{sg:zhang13}
Y.~Zhang, N.~Gatsis, and G.~Giannakis, ``Robust energy management for
  microgrids with high-penetration renewables,'' \emph{Sustainable Energy, IEEE
  Transactions on}, vol.~PP, no.~99, pp. 1--10, 2013.

\bibitem{opt:admm10}
S.~Boyd, N.~Parikh, E.~Chu, B.~Peleato, and J.~Eckstein, ``Distributed
  optimization and statistical learning via the alternating direction method of
  multipliers,'' \emph{Foundations and Trends in Machine Learning}, vol.~3,
  no.~1, pp. 1--122, 2010.

\bibitem{sg:calafiore06}
G.~Calafiore and M.~Campi, ``The scenario approach to robust control design,''
  \emph{Automatic Control, IEEE Transactions on}, vol.~51, no.~5, pp. 742--753,
  2006.

\bibitem{book:stoopt1}
A.~Shapiro, D.~Dentcheva, and A.~Ruszczynski, \emph{Lectures on Stochastic
  Programming: Modeling and Theory}.\hskip 1em plus 0.5em minus 0.4em\relax
  SIAM, 2009.

\bibitem{sg:yang13storagecv}
P.~Yang and A.~Nehorai, ``Hybrid energy storage and generation planning with
  large renewable penetration,'' in \emph{IEEE International Workshop on
  Computational Advances in Multi-Sensor Adaptive Processing}, Dec. 2013, pp.
  1--4.

\bibitem{sg:storageopt10}
EPRI, ``Electricity energy storage technology options: A white paper primer on
  applications, costs, and benefits,'' EPRI, Palo Alto, Tech. Rep., 2010.

\bibitem{sg:solardata}
\BIBentryALTinterwordspacing
\emph{National Solar Radiation Data Base}. [Online]. Available:
  \url{http://rredc.nrel.gov/solar/old_data/nsrdb/}
\BIBentrySTDinterwordspacing

\bibitem{sg:solarradman}
S.~Wilcox, \emph{National Solar Radiation Database 1991 –- 2010 Update: User's
  Manual}, Aug. 2012.

\bibitem{sg:solartilt}
\BIBentryALTinterwordspacing
\emph{Optimum Tilt of Solar Panels}. [Online]. Available:
  \url{http://www.macslab.com/optsolar.html}
\BIBentrySTDinterwordspacing

\bibitem{sg:renewableAsses07}
EPRI, ``Renewable energy technical assessment guide - {TAG}-{RE}:2006,'' EPRI,
  Palo Alto, Tech. Rep., 2007.

\bibitem{sg:ercotload}
\BIBentryALTinterwordspacing
\emph{ERCOT Hourly Load Data Archive}. [Online]. Available:
  \url{http://www.ercot.com/gridinfo/load/load_hist/}
\BIBentrySTDinterwordspacing

\bibitem{cvx_n}
M.~Grant and S.~Boyd, ``{CVX}: Matlab software for disciplined convex
  programming, version 2.0 beta,'' \url{http://cvxr.com/cvx}, Sep. 2012.

\end{thebibliography}

\end{document}